\documentclass[10pt]{article}

\usepackage{amsmath,amssymb,amsthm}
\providecommand{\R}{\mathbb R}
\providecommand{\E}{\mathsf E}
\providecommand{\prob}{\mathsf P}
\providecommand{\e}{\mathrm e}

\newtheorem{thm}{Theorem}
\newtheorem{lem}[thm]{Lemma}
\newtheorem{prop}[thm]{Proposition}

\theoremstyle{remark}
\newtheorem*{rem}{Remark}

\title{Regularization in $L_1$ for the Ornstein--Uhlenbeck semigroup}
\author{
Joseph Lehec
\footnote{CEREMADE (UMR CNRS 7534) Universit\'e Paris--Dauphine.}
}
\begin{document}
\maketitle

\begin{abstract}
Let $\gamma_n$ be the standard 
Gaussian measure on $\mathbb R^n$
and let $(Q_t)$
be the Ornstein--Ulhenbeck semigroup. 
Eldan and Lee recently established that
for every non--negative function $f$ 
of integral $1$ and any time $t$ 
the following tail inequality holds true:
 \[
\gamma_n ( \{ Q_t f > r \} )
\leq C_t \, \frac{ (\log \log r)^4 }{ r \sqrt{ \log r} } , 
\quad \forall r>1
\]
where $C_t$ is a constant depending on $t$
but not on the dimension. 
The purpose of the present paper
is to simplify parts of their argument 
and to remove the $(\log \log r)^4$ factor.  
\end{abstract}
\section{Introduction}
Let $\gamma_n$ be the standard Gaussian measure on 
$\R^n$ and let $(Q_t)$ be the Ornstein--Ulhenbeck 
semigroup: 
for every test function $f$
\begin{equation}
\label{OU}
Q_t f ( x ) 
= \int_{\R^n} f \left( \e^{-t} x + \sqrt{ 1 - \e^{-2t} } \, y \right) \, \gamma_n (dy) .
\end{equation}
Nelson~\cite{nelson} established that
if $p>1$ and $t>0$ then $Q_t$ is a
contraction from $L_p (\gamma_n)$ to $L_q (\gamma_n)$ for some 
$q >p$, namely for 
\[
q  =  1 + \e^{2t}  (p-1) . 
\]
The semigroup $(Q_t)$ is said to be 
\emph{hypercontractive}. This turns out to be equivalent 
to the logarithmic Sobolev inequality
(see the classical article by Gross~\cite{gross}). 
In this paper we establish a regularity
property of $Q_t f$ assuming only that $f$ is in 
$L^1 ( \gamma_n)$. 

Let $f$ be a non--negative function satisfying 
\[
\int_{\R^n} f \, d \gamma_n = 1 ,
\]
and let $t>0$. Since $Q_t f \geq 0$ and  
\[
\int_{\R^n} Q_t f \, d\gamma_n = \int_{\R^n} f \, d\gamma_n = 1,
\]
Markov inequality gives
\[
\gamma_n \left( \{ Q_t f \geq r \} \right) \leq \frac 1 r ,
\]
for all $r\geq 1$. 
Now Markov inequality is only sharp 
for indicator functions and $Q_t f$ 
cannot be an indicator function, 
so it may be the case that this inequality 
can be improved. More precisely 
one might conjecture that
for any fixed $t>0$ 
(or at least for $t$ large enough)
there exists a function 
$\alpha$ satisfying
\[
\lim_{r\to +\infty} \alpha (r) = 0 
\]
and
\begin{equation}\label{conj-tal}
\gamma_n \left( \{ Q_t f \geq r \} \right) 
\leq \frac { \alpha (r) } r ,
\end{equation}
for every $r\geq 1$ and for every non--negative function 
$f$ of integral $1$. 
The function $\alpha$ should 
be independent of the dimension $n$, 
just as the hypercontractivity 
result stated above. 
Such a phenomenon was actually conjectured
by Talagrand in~\cite{talagrand}
in a slightly different context. He conjectured
that the same inequality holds true
when $\gamma_n$ is replaced by 
the uniform measure on the discrete cube $\{-1,1\}^n$
and the Orstein--Uhlenbeck semigroup
is replaced by the semigroup associated to the
random walk on the discrete cube.
The Gaussian version of the conjecture 
would follow from Talagrand's discrete version
by the central limit theorem. In this 
paper we will only focus on the Gaussian case. 

In~\cite{ball-etc}, Ball, Barthe, Bednorz,
Oleszkiewicz and Wolff 
showed that in dimension $1$
the inequality~\eqref{conj-tal}
holds with decay
\[
\alpha (r) = \frac{ C }{ \sqrt{ \log r } } ,
\]
where the constant $C$ depends on the time 
parameter $t$. 
Moreover the authors provide an example 
showing that the $1/\sqrt{\log r}$ decay is sharp. 
They also have a result in higher dimension 
but they loose a factor $\log \log r$ and,
more importantly, their constant $C$ then tends
to $+\infty$ (actually exponentially fast) 
with the dimension. The deadlock was broken 
recently by Eldan and Lee who showed
in~\cite{eldan-lee} that~\eqref{conj-tal}
holds with function 
\[
\alpha (r) = C \, 
\frac{ ( \log  \log r )^4 }{ \sqrt{ \log r }} ,
\]
with a constant $C$ that is independent 
of the dimension. Again up to the $\log \log$ 
factor the result is optimal. 

In this article we revisit the argument
of Eldan and Lee. We shall simplify 
some steps of their proof and short cut some 
others. As a result, we are able to remove
the extra $\log \log$ factor. We would like 
to make clear though that this note does 
not really contain any new idea 
and that the core of our 
argument is all Eldan and Lee's. 
\section{Main results}
Recall that $\gamma_n$ is the standard Gaussian measure
and that $(Q_t)$ is the Ornstein--Uhlenbeck semigroup,
defined by~\eqref{OU}. Here is our main result. 
\begin{thm}\label{thm-main}
Let $f$ be a non--negative function on $\R^n$ 
satisfying $\int_{\R^n} f \, d\gamma_n = 1$ 
and let $t>0$. Then for every $r> 1$
\[
\gamma_n \left( \{ Q_t f > r \} \right) 
\leq  C\, \frac {  \max ( 1 , t^{-1} ) } {r \sqrt{\log r} } ,
\]
where $C$ is a universal constant.
\end{thm}
As in Eldan and Lee's paper,
the Ornstein--Ulhenbeck
semigroup only plays
a r\^ole through the following 
lemma. 
\begin{lem}\label{lemm}
Let $f \colon \R^n \to \R_+$. For every $t>0$, we have
\[
\nabla^2 \log ( Q_t f) \geq - \frac 1 {2t} \, \mathrm{id} ,  
\]
pointwise. 
\end{lem}  
\begin{proof}
This really straightforward:
observe that~\eqref{OU} can be rewritten as
\[
Q_t f ( x )  = ( f \ast g_{1-\rho} ) ( \rho \, x ) ,
\]
where $\rho= \e^{-t}$ and $g_{1-\rho}$ is the
density of the Gaussian measure with mean $0$ 
and covariance $(1-\rho) \mathrm{id}$. Then 
differentiate twice and use the Cauchy--Schwarz
inequality. Details are left to the reader. 
\end{proof}
What we actually prove
is the following, 
where $Q_t$ does not 
appear anymore. 
\begin{thm}\label{thm1}
Let $f$ be a positive function on $\R^n$ 
satisfying $\int_{\R^n} f \, d\gamma_n = 1$. 
Assume that $f$ is smooth and satisfies
\begin{equation}\label{hyp1}
\nabla^2 \log f \geq - \beta \, \mathrm{id} 
\end{equation}
pointwise, for some $\beta\geq 0$.
Then for every $r> 1$
\[
\gamma_n \left( \{ f > r \} \right) 
\leq \frac { C\, \max ( \beta , 1 ) } { r \sqrt{\log r} } ,
\]
where $C$ is a universal constant. 
\end{thm}
Obviously, Theorem~\ref{thm1}
and Lemma~\ref{lemm} altogether
yield Theorem~\ref{thm-main}.

Let us comment on the optimality
of Theorem~\ref{thm-main} and Theorem~\ref{thm1}. 
In dimension $1$, consider the function 
\[
f_\alpha (x) = \e^{ \alpha x - \alpha^2 /2 } . 
\]
Observe that $f_\alpha \geq 0$ and that $\int_\R f_\alpha \, d\gamma = 1$. 
Note that for every $t\geq 1$ we have
\[
\gamma_1 \left( [t,+\infty) \right) \geq \frac{ c\, \e^{ -t^2/2 } } t , 
\]
where $c$ is a universal constant.
So if $\alpha >0$ and $r\geq \e$ then
\[
\gamma_1 \left( \{ f_\alpha \geq r \} \right) \geq 
\frac{ c\, \exp \left(  - \frac 12 \left( \frac {\log r} \alpha + \frac \alpha 2 \right)^2 \right) }
 { \frac {\log r} \alpha + \frac \alpha 2 } . 
\]
Choosing $\alpha = \sqrt{ 2 \log r}$ we get
\[
\gamma_1 \left( \{ f_\alpha \geq r \} \right) \geq  
\frac{ c' }{ r \sqrt{ \log r} } . 
\]
Since $(\log f_\alpha)'' = 0$ this shows 
that the dependence in $r$ in Theorem~\ref{thm1}
is sharp. Actually this example also shows that 
the dependence in $r$ in Theorem~\ref{thm-main}
is sharp. Indeed, it is easily seen that
\[
Q_t f_\alpha = f_{\alpha \e^{-t}} ,
\]
for every $\alpha\in \R$ and $t>0$. This implies that
$f_\alpha$ always belongs to the image $Q_t$.  
Of course, this example also works in higher dimension: 
just replace $f_\alpha$ by 
\[
f_u (x) = \e^{ \langle u , x \rangle - \vert u \vert^2 /2 } 
\]
where $u$ belongs to $\R^n$. 
\begin{thm}\label{thm2}
Let $X$ be a random vector having density $f$ 
with respect to the Gaussian measure, and assume that
$f$ satisfies~\eqref{hyp1}. Then for every $r>1$ 
\[
\prob \left( f(X) \in (r,\e\,r] \right) 
\leq C \, \frac {\max (\beta,1)} {\sqrt{\log r}} . 
\]
\end{thm}
Theorem~\ref{thm2} easily yields Theorem~\ref{thm1}.
\begin{proof}[Proof of Theorem~\ref{thm1}]
Let $G$ be standard Gaussian 
vector on $\R^n$
and let $X$ 
be a random vector having density $f$ 
with respect to $\gamma_n$. Then using 
Theorem~\ref{thm2}
\[
\begin{split}
\prob [ f(G) > r ] 
& = \sum_{k=0}^{+\infty} \prob \left( f(G) \in (e^k r , e^{k+1} r ] \right) \\
& \leq  \sum_{k=0}^{+\infty}  
 (\e^k r)^{-1} \E \left[ f(G) \, \mathbf 1_{ \{ f(G) \in (e^k r , e^{k+1} r ] \} } \right] \\
& =  \sum_{k=0}^{+\infty}  
 (\e^k r)^{-1} \prob \left( f(X) \in (e^k r , e^{k+1} r ] \right) \\
& \leq  \sum_{k=0}^{+\infty}  
 (\e^k r)^{-1} \, C \,  \frac { \max ( \beta , 1 )  }{ \sqrt{ \log( \e^k r ) } }  \\
& \leq C \, \frac \e {\e-1} \, \frac 1 r \,\frac { \max(\beta,1) }{ \sqrt{ \log r } }  ,
\end{split}
\]
which is the result. 
\end{proof}
The rest of the note is devoted to the proof of Theorem~\ref{thm2}. 
\section{Preliminaries: the stochastic construction}
Let $\mu$ be a probability measure on $\R^n$ 
having density $f$ with respect to
the Gaussian measure.
We shall assume that $f$ is
bounded away from $0$, that $f$ is $\mathcal C^2$
and that $\nabla f$ and $\nabla^2 f$ are bounded. 
A simple density argument shows that 
we do not lose generality by adding these 
technical assumptions.

Eldan and Lee's argument is based on a stochastic
construction which we describe now.
Let $(B_t)$ be a standard $n$--dimensional Brownian motion
and let $(P_t)$ be the associated semigroup:
\[
P_t h (x) = \E [ h ( x+B_t) ] , 
\]
for all test functions $h$.
Note that $(P_t)$ is the heat semigroup, 
not the Ornstein--Ulhenbeck semigroup.
Consider the stochastic differential equation
\begin{equation}\label{defX}
\left\{
\begin{array}{l}
X_0  = 0 \\
d X_t  = d B_t +  
\nabla \log ( P_{1-t} f ) (X_t) \, dt , \quad t\in [0,1] .
\end{array}
\right.
\end{equation}
The technical assumptions made on $f$ 
insure that the map
\[
x \mapsto \nabla \log P_{1-t} f (x)
\]  
is Lipschitz, with a Lipschitz norm that does not
depend on $t\in [0,1]$. So the equation~\eqref{defX}
has a strong solution $(X_t)$. 
In our previous work~\cite{lehec}
we study the process $(X_t)$ in details 
and we give some applications 
to functional inequalities. 
Let us recap here some of these 
properties and refer to~\cite[section 2.5]{lehec}
for proofs. Recall that if $\mu_1,\mu_2$
are two probability measures, the relative 
entropy of $\mu_1$ with respect to $\mu_2$ 
is defined by
\[
\mathrm H ( \mu_1 \mid \mu_2 ) 
= \int \log \left( \frac{ d \mu_1 }{ d \mu_2 } \right) \, d \mu_1 ,
\]  
if $\mu_1$ is absolutely continuous with respect to $\mu_2$ 
(and $\mathrm H ( \mu_1 \mid \mu_2 ) = +\infty$ otherwise).
Also in the sequel we call \emph{drift} any process
$(u_t)$ taking values in $\R^n$ which is
adapted to the natural filtration of $(B_t)$
(this means that $u_t$ depends only on 
$(B_s)_{s\leq t}$) and satisfies
\[
\int_0^1 \vert u_t \vert^2 \, ds < +\infty. 
\]
Let $(v_t)$ be the drift
\[
v_t = \nabla \log P_{1-t} f ( X_t ) .
\]
Using It\^o's formula it is easily seen 
that
\[
d \log P_{1-t} f (X_t) = 
\langle v_t , d B_t \rangle + \frac 12 \vert v_t \vert^2 \, dt. 
\]
Therefore, for every $t\in [0,1]$
\begin{equation}\label{fvt}
P_{1-t} f ( X_t ) = 
\exp \left( \int_0^t \langle v_s , dB_s \rangle 
+ \frac 12 \int_0^t \vert v_s \vert^2 \,  ds \right) . 
\end{equation}
Combining this with the Girsanov change of measure theorem
one can show that the random vector $X_1$ has law $\mu$ 
(again we refer to~\cite{lehec} for details).
Moreover we have the equality
\begin{equation}\label{entropy1}
\mathrm H ( \mu \mid \gamma_n ) 
= \frac 12 \E \left[ \int_0^T \vert v_s \vert^2 \, ds \right] .  
\end{equation}
Also, if $(u_t)$ is any drift
and if $\nu$ is the law of 
\[
B_1 + \int_0^1 u_t \, dt ,
\]
then 
\begin{equation}\label{entropy11}
\mathrm H ( \nu \vert \gamma_n ) 
\leq \frac 12 \E \left[ \int_0^T \vert u_s \vert^2 \, ds \right] .  
\end{equation}
So the drift $(v_t)$ is in some sense optimal. 
Lastly, and this will play a crucial r\^ole in the sequel,
the process $(v_t)$ is a martingale.

Eldan and Lee introduce a perturbed version 
of the process $(X_t)$, which we now describe. 
From now on we fix $r> 1$ and we let
\[
T = \inf \{ t \in [0,1] , \ P_{1-t} f ( X_t ) > r \} \ \wedge 1  
\]
be the first time the process $(P_{1-t} f (X_t ))$ hits the value $r$ 
(with the convention that $T =1$ if it does not ever reach $r$). 
Now given $\delta >0$ we let $(X_t^\delta)$ be the process defined by
\[
X_t^\delta = X_t + \delta \int_0^{T\wedge t} v_s \, ds .
\]
Note that this perturbed process is still of the form Brownian motion plus drift:
\[
X^\delta_t = B_t + \int_0^t (1 + \delta \mathbf{1}_{ \{ s\leq T\} } ) v_s \, ds. 
\]
So letting $\mu^\delta$ be the law of $X_1^\delta$
and using~\eqref{entropy11} we get
\begin{equation}\label{entropy2}
\begin{split}
\mathrm H ( \mu^\delta \mid \gamma ) 
& \leq \frac 12 \E \left[ \int_0^1 (1+ \delta \, \mathbf 1_{ \{s\leq T\} } )^2 \vert v_s \vert^2 \, ds \right] \\
& =  \frac 12 \E \left[ \int_0^1 \vert v_s \vert^2 \, ds \right] 
+ \left( \delta + \frac{\delta^2}2 \right) \, \E\left[ \int_0^T \vert v_s \vert^2 \, ds \right] .
\end{split}
\end{equation}
\section{Proof of the main result}
The proof can be decomposed 
into two steps. Recall that $r$ is fixed 
from the beginning and that $X_1^\delta$ actually
depends on $r$ through the stopping time $T$. 

The first step is to prove that 
if $\delta$ is small then $\mu$ and $\mu^\delta$
are not too different. 
\begin{prop}\label{prop1}
Assuming~\eqref{hyp1}, we have 
\[
d_{TV} ( \mu , \mu^\delta ) \leq \delta \sqrt{ (\beta+1) \log r } ,
\] 
for every $\delta>0$,
where $d_{TV}$ denotes the total variation distance. 
\end{prop}
The second step is to argue 
that $f(X_1^\delta)$ tends to be bigger 
than $f(X_1)$. An intuition for this property 
is that the difference between $X_1^\delta$ and 
$X_1$ is somehow in the direction of $\nabla f (X_1 )$.
\begin{prop}\label{prop2}
Assuming~\eqref{hyp1}, we have 
\[
\prob \left( f(X_1^\delta) \leq r^{1+2\delta} \e^{-4} \right)
\leq \prob ( f (X_1) \leq r ) + (\beta+4) \delta^2 \log (r) ,
\] 
for all $\delta>0$. 
\end{prop}
\begin{rem}
Note that both propositions use 
the convexity hypothesis~\eqref{hyp1}. 
\end{rem}
It is now very easy to prove Theorem~\ref{thm2}.
Since $X_1$ has law $\mu$, all we need to prove is
\[
\prob ( f(X_1) \in (r,\e r ] ) \leq C \, \frac{ \max(\beta,1) }{ \sqrt{ \log r }} . 
\]
We choose 
\[
\delta = \frac 5 {2 \log r} . 
\]
For this value of $\delta$, Proposition~\ref{prop2} gives
\[
\prob ( f ( X_1^\delta ) \geq \e \, r ) \leq \prob ( f ( X_1 ) \leq  r) + 
\frac {25} 4 \, \frac{\beta+4}{\log r} ,
\]
whereas Proposition~\ref{prop1} yields
\[
\begin{split}
\prob ( f ( X_1 ) \leq \e r ) 
& \leq \prob ( f(X_1^\delta) \leq \e r ) + d_{TV} ( \mu , \mu^\delta ) \\
& \leq \prob ( f(X_1^\delta) \leq \e r ) + \frac 52 \left( \frac{ \beta+1 }{ \log r } \right)^{1/2} .
\end{split}
\]
Combining the two inequalities we obtain 
\[
\prob ( f ( X_1 ) \leq \e \,  r  )
\leq 
\prob ( f ( X_1 ) \leq r ) +  
C \, \frac{ \max (\beta,1) }{ \sqrt{ \log r } } ,
\]
which is the result. 
\begin{rem}
We actually prove the slightly stronger statement:
\[
\prob ( f ( X_1 ) \in ( r , \e r ] )  
\leq C \max \left( \frac {\max(\beta,1)} {\log r } , 
\left( \frac {\max(\beta,1)} {\log r} \right)^{1/2} \right) . 
\]
\end{rem}
\section{Proof of the total variation estimate}
We actually bound the relative entropy of
$\mu^\delta$ with respect to $\mu$. 
Recall that $\log f$ is assumed to be weakly convex:
there exists $\beta\geq 0$ such that 
\begin{equation}\label{hyp11}
\nabla^2 \log f \geq - \beta \, \mathrm{id} ,
\end{equation}
pointwise. 
\begin{prop}
Assuming~\eqref{hyp11}, we have
\[
\mathrm H ( \mu^\delta \mid \mu )
\leq \delta^2 ( \beta +1 ) \log r , 
\]
for all $\delta>0$.
\end{prop}
This yields Proposition~\ref{prop1}
by Pinsker's inequality. 
\begin{proof}
Observe that 
\begin{equation}\label{step1}
\mathrm H ( \mu^\delta \mid \mu ) = 
\mathrm H ( \mu^\delta \mid \gamma ) - \int_{\R^n} \log ( f ) \, \, d\mu^\delta .  
\end{equation}
Now~\eqref{hyp11} gives
\begin{equation}\label{convex}
\begin{split}
\log ( f )( X_1^\delta ) & \geq \log f ( X_1 ) 
+ \langle \nabla \log f ( X_1 ) , X_1^\delta - X_1 \rangle
- \frac \beta 2 \vert X_1^\delta - X_1 \vert^2 , \\
& \geq \log f ( X_1 ) + \delta \int_0^T \langle v_1 , v_s \rangle \, ds 
 - \frac{\beta \delta^2 } 2 \int_0^T \vert v_s \vert^2 \, ds , 
\end{split}
\end{equation}
almost surely. 
We shall use this inequality 
several times in the sequel. 
Recall that $X_1$ has law $\mu$ 
and that $X_1^\delta$ has law $\mu^\delta$.
Taking expectation in the previous inequality
and using~\eqref{step1} we get
\[
\begin{split}
\mathrm H ( \mu^\delta \mid \mu ) 
& \leq \mathrm H ( \mu^\delta \mid \gamma )  - \mathrm H ( \mu \mid \gamma ) \\
& - \delta \, \E \left[ \int_0^T \langle v_1 , v_s \rangle \, ds \right]
+  \frac {\beta \delta^2 } 2 \, 
\E \left[ \int_0^T \vert v_s \vert^2 \, ds \right] .
\end{split}
\]
Together with~\eqref{entropy1} and~\eqref{entropy2}
we obtain
\[
\mathrm H ( \mu^\delta \mid \mu ) 
\leq 
- \delta \, \E \left[ \int_0^T \langle v_1-v_s , v_s \rangle \, ds \right]
+  \frac {(1+\beta) \delta^2 } 2 \, 
\E \left[ \int_0^T \vert v_s \vert^2 \, ds \right] .
\]
Now since $(v_t)$ is a martingale
and $T$ a stopping time we have
\[
\E \left[ \langle v_1 , v_s \rangle \, \mathbf 1_{\{s\leq T\}} \right]
=  \E \left[ \vert v_s \vert^2  \mathbf 1_{\{s\leq T\}} \right]
\]
for all time $s\leq 1$. This shows that the first term 
in the previous inequality is $0$. To bound the second term,
observe that the definition of $T$ and the equality~\eqref{fvt}
imply that 
\[
\int_0^T \langle v_s , dB_s \rangle 
+ \frac 12 \int_0^T \vert v_s \vert^2 \, ds \leq \log r,
\]
almost surely. Since $(v_t)$ is a bounded drift, the process
$(\int_0^t \langle v_s, dB_s \rangle)$
is a martingale. Now $T$ is a bounded stopping time,
so by the optional stopping theorem
\[
\E \left[ \int_0^T \langle v_s , d B_s\rangle \right] = 0.
\]
Therefore, taking expectation in the previous inequality yields
\[
\E \left[ \int_0^T \vert v_s \vert^2 \, ds \right] \leq 2 \log r ,
\]
which concludes the proof. 
\end{proof}
\section{Proof of Proposition~\ref{prop2}}
The goal is to prove that 
\[
\prob \left( f(X_1^\delta) \leq r^{1+2\delta} \e^{-4} \right)
\leq \prob ( f (X_1) \leq r ) + \delta^2 (\beta+4) \log r.
\]
Obviously
\[
\prob \left( f(X_1^\delta) \leq r^{1+2\delta} \e^{-4} \right)
\leq
\prob ( f (X_1) \leq r ) +
\prob \left( f ( X_1^\delta) \leq r^{1+2\delta} \e^{-4} ; \ f(X_1) > r \right) 
\]
Now recall the inequality~\eqref{convex} coming for the weak convexity
of $\log f$ and rewrite it as
\[
\log f(X_1^\delta) \geq K_1 + 2 \delta K_T + Y
\]
where $(K_t)$ is the process defined by
\[
K_t = \log ( P_{1-t} ) (f) (X_t ) 
= \int_0^t \langle v_s , dB_s \rangle + \frac 12 \int_0^t \vert v_s \vert^2 \, ds ,
\]
and $Y$ is the random variable
\[
Y = -2\delta \int_0^T \langle v_s , d B_s \rangle 
 + \delta \int_0^T \langle v_1 - v_s  , v_s \rangle \, ds 
 - \frac{\beta \delta^2} 2 \int_0^T \vert v_s \vert^2 \, ds .  
\]
Recall that the stopping time $T$ is the first time the process
$(K_t)$ exceeds the value $\log r$ if it ever does,
and $T=1$ otherwise. In particular, if
\[
K_1 = \log f (X_1) > \log r 
\]
then $K_T = \log r$. So if $f(X_1) > r$
then 
\[
f(X_1^\delta) > r^{1+2\delta} \, \e^Y . 
\]
Therefore 
\[
\prob \left( f ( X_1^\delta) \leq r^{1+2\delta} \e^{-4} ; \ f(X_1) > r \right) 
\leq 
\prob ( Y \leq -4 ).
\]
So we are done if we can prove that 
\begin{equation}\label{ineqY}
\prob ( Y \leq -4 ) \leq 
(\beta+4) \delta^2 \log r . 
\end{equation}
There are three terms in the definition 
of $Y$. The problematic one is
\[
\delta \int_0^T \langle v_1 - v_s  , v_s \rangle \, ds . 
\] 
We know from the previous section 
that it has expectation $0$.
A natural way to get a deviation bound would 
be to estimate its second moment but it 
is not clear to us how to do this. 
Instead we make an complicated detour.
\begin{lem}
\label{bizarre}
Let $Z$ be an integrable random 
variable satisfying $\E [\e^Z]\leq 1$.
Then
\[
\prob ( Z \leq -2 ) \leq - \E[Z]  . 
\] 
\end{lem}
\begin{rem} Note that $\E [ Z ] \leq 0$ 
by Jensen's inequality.  
\end{rem}
\begin{proof}
Simply write 
\[
\begin{split}
\E\left[ \e^Z \right] 
& \geq \E\left[ \e^Z \, \mathbf 1_{ \{ Z>-2\} }\right]  \\
& \geq \E\left[ (Z+1) \, \mathbf 1_{ \{ Z>-2\} }\right] \\
& = \E [Z] - \E\left[ Z \, \mathbf 1_{ \{ Z\leq -2\} }\right] + 1 - \prob ( Z \leq - 2 ) \\
& \geq \E [ Z ] + \prob ( Z \leq -2 ) + 1 . 
\end{split}
\] 
So if $\E [ \e^Z ] \leq 1$ then $\prob ( Z \leq -2 ) \leq -\E[Z]$.
\end{proof}
\begin{lem}\label{lemZ}
Let $Z$ be the variable
\[
Z =  - \delta \int_0^T \langle v_s , d B_s \rangle 
+ \delta \int_0^T \langle v_1 - v_s , v_s \rangle \, ds
- \frac{(\beta +1) \delta^2} 2 \int_0^T \vert v_s \vert^2 \, ds . 
\]
Then 
\[
\prob ( Z \leq -2 ) \leq \delta^2 (\beta +1) \log r . 
\]
\end{lem}
\begin{proof}
As we have seen before 
the first two terms in the 
definition of $Z$ have expectation $0$
and 
\[
\E [ Z ] = -  \frac{(\beta+1) \delta^2} 2 \, 
\E\left[ \int_0^T \vert v_s \vert^2 \, ds \right]
\geq -  \delta^2  (\beta+1) \log r . 
\]
By Lemma~\ref{bizarre} it is enough 
to show that $\E [ \e^{Z} ] \leq 1$.
To do so, we use the Girsanov change of measure formula. 
The process $(X_t^\delta)$
is of the form Brownian motion plus drift: 
\[
\begin{split}
X_t^\delta 
& = X_t + \delta \int_0^{T\wedge t} v_s \, ds \\
& =  B_t + \int_0^t (1 + \delta \mathbf{1}_{ \{ s\leq T\} } ) v_s \, ds ,
\end{split}
\]
Note also that the drift term is bounded. 
Therefore, Girsanov's formula applies,
see for instance~\cite[chapter 6]{liptser}
(beware that the authors oddly use the letter $M$
to denote expectation). 
The process $(D^\delta_t)$ defined by
\[
D^\delta_t 
= \exp \left( - \int_0^t (1+ \delta \mathbf{1}_{\{s\leq T\}} ) \langle v_s , d B_s \rangle 
- \frac 12 \int_0^t \left\vert (1+\mathbf{1}_{\{s\leq T\}}) v_s \right\vert^2 \, ds\right)
\]
is a non-negative martingale of expectation $1$ and 
under the measure $\mathsf Q^{\delta}$ defined by
\[
d \mathsf Q^\delta = D^\delta_1 \, d \mathsf P
\]
the process $(X_t^\delta)$ is a standard Brownian motion. 
In particular 
\[
\E [ f ( X_1^\delta ) D_1^\delta ]
= \E [ f ( B_1 ) ] = 1 .
\]
Now we use inequality~\eqref{convex} once again.
A tedious but elementary computation shows that 
it gives exactly
\[
f( X_1^\delta ) \, D_1^\delta \geq \e^{Z} . 
\]
Therefore $\E [ \e^{Z} ] \leq 1$, which concludes the proof. 
\end{proof}
We now prove inequality~\eqref{ineqY}. The idea being that
the annoying term in $Y$ is handled by the previous lemma. 
Observe that 
\[
Y = Z - \delta \int_0^T \langle v_s , d B_s \rangle - \frac {\delta^2} 2 \int_0^T \vert v_s \vert^2 \, ds.  
\]
So
\[
\begin{split}
\prob ( Y \leq - 4 ) \leq \prob ( Z \leq - 2 ) 
& + \prob \left( \delta \int_0^T \langle v_s , d B_s \rangle \geq 1 \right) \\  
& +  \prob \left( \frac{\delta^2}2 \int_0^T \vert v_s \vert^2 \, ds  \geq 1 \right) .  
\end{split}
\]
Recall that $\int_0^T \langle v_s , d B_s \rangle$ has mean $0$ and observe that 
\[
\E \left[ \left( \delta \int_0^T \langle v_s , d B_s \rangle \right)^2 \right]
= \delta^2 \E \left[ \int_0^T \vert v_s \vert^2 \, ds \right]
\leq 2 \delta^2 \log r. 
\]
So by Tchebychev inequality
\[
\prob \left( \delta \int_0^T \langle v_s , d B_s \rangle \geq 1 \right) \leq 2 \delta^2 \log r. 
\]
Similarly by Markov inequality
\[
\prob \left( \frac{\delta^2}2 \int_0^T \vert v_s \vert^2 \, ds  \geq 1 \right) \leq \delta^2 \log r .
\]
Putting everything together we get~\eqref{ineqY}, which concludes the proof. 


\begin{thebibliography}{plain}

\bibitem{ball-etc}
Ball, K.; Barthe, F.; Bednorz, W.; Oleszkiewicz, K.; Wolff, P. 
\emph{$L_1$--smoothing for the Ornstein--Uhlenbeck semigroup}.
Mathematika 59 (2013), no. 1, 160--168.

\bibitem{eldan-lee}
Eldan,~R.; Lee,~J.
\emph{Regularization under diffusion and anti--concentration of temperature}.
arXiv:1410.3887. 

\bibitem{gross}
Gross,~L.
\emph{Logarithmic Sobolev inequalities}.
Amer. J. Math. 97 (1975), no. 4, 1061--1083. 

\bibitem{lehec}
Lehec,~J.
\emph{Representation formula for the entropy and functional inequalities.}
Ann. Inst. Henri Poincar\'e Probab. Stat. 49 (2013), no. 3, 885--899.

\bibitem{liptser}
Liptser,~R.; Shiryaev,~A. 
Statistics of random processes. 
Vol I, general theory. 2nd edition. 
\emph{Stochastic Modelling and Applied Probability}. 
Springer--Verlag, Berlin, 2001.

\bibitem{nelson}
Nelson,~E.
\emph{The free Markoff field}.
J. Functional Analysis 12 (1973), 211--227.

\bibitem{talagrand} 
Talagrand,~M.
\emph{A conjecture on convolution operators, 
and a non-Dunford-Pettis operator on $L_1$}.
Israel J. Math. 68 (1989), no. 1, 82--88. 
\end{thebibliography}
\end{document}